\documentclass[12pt, reqno]{amsart}
\usepackage{fullpage}

\usepackage{amsmath,amstext,amsthm,amssymb,graphicx}
\usepackage{enumitem,verbatim}
\usepackage[backref,colorlinks=true,linkcolor=red,citecolor=blue]{hyperref}
\usepackage{cleveref}
\usepackage[alphabetic,backrefs,lite]{amsrefs}
\usepackage{tikz-cd}
\usepackage{multirow}
\usepackage{float}
\usepackage{xspace}

\newcommand{\ve}{\varepsilon}

\newcommand{\commented}[1]{}

\newcommand{\A}{\mathbb{A}}

\newcommand{\F}{\mathbb{F}}

\newcommand{\PP}{\mathbb{P}}
\newcommand{\Q}{\mathbb{Q}}

\newcommand{\Z}{\mathbb{Z}}

\newcommand{\calA}{{\mathcal A}}

\DeclareMathOperator{\Res}{Res}
\DeclareMathOperator{\Cor}{Cor}

\DeclareMathOperator{\Nm}{Nm}

\DeclareMathOperator{\Br}{Br}
\DeclareMathOperator{\inv}{inv}
\DeclareMathOperator{\ev}{ev}

\DeclareMathOperator{\Gal}{Gal}

\let\badexists\exists
\renewcommand{\exists}{\,\,\badexists\,\,}
\let\badforall\forall
\renewcommand{\forall}{\,\,\badforall\,\,}
\let\badlim\lim
\renewcommand{\lim}{\badlim\limits}
\let\badlimsup\limsup
\renewcommand{\limsup}{\badlimsup\limits}
\let\badliminf\liminf
\renewcommand{\liminf}{\badliminf\limits}
\let\badsum\sum
\renewcommand{\sum}{\badsum\limits}
\let\badprod\prod
\renewcommand{\prod}{\badprod\limits}
\let\badbigcap\bigcap
\renewcommand{\bigcap}{\badbigcap\limits}
\let\badbigcup\bigcup
\renewcommand{\bigcup}{\badbigcup\limits}

\theoremstyle{plain}
\newtheorem{Theorem}{Theorem}[section]
\newtheorem{Lemma}[Theorem]{Lemma}
\newtheorem{Corollary}[Theorem]{Corollary}

\theoremstyle{definition}

\theoremstyle{remark}
\newtheorem{Remark}[Theorem]{Remark}

\title{On the Hasse Principle for Conic Bundles over Even Degree Extensions}
\author{Sam Roven}
\author{Alexander Wang}

\address{Department of Mathematics \& Statistics, University of San Francisco, San Francisco, CA 94117, USA}
\email{smroven@usfca.edu}
\urladdr{https://www.samroven.com}

\address{Department of Mathematics, University of Washington, Seattle, WA 98195, USA}
\email{awaw@uw.edu}
\urladdr{https://math.washington.edu/\~{}awaw}

\newcommand{\CT}{Colliot-Th\'el\`ene\xspace}

\begin{document}

\begin{abstract}
Let $k$ be a number field and let $\pi \colon X \rightarrow\PP_k^1$ be a smooth conic bundle. We show that if $X/k$ has four geometric singular fibers and either $X(\A_k)\neq \emptyset$ or $X/k$ has non-trivial Brauer group, then $X$ satisfies the Hasse principle over any even degree extension $L/k$. Furthermore for arbitrary $X$ we show that, conditional on Schinzel's hypothesis, $X$ satisfies the Hasse principle over all but finitely many quadratic extensions of $k$. We prove these results by showing the Brauer-Manin obstruction vanishes and then apply fibration method results of \CT, following \CT and Sansuc.
\end{abstract}

\maketitle

\section{Introduction}

Let $X \rightarrow \PP^1_k$ be a smooth conic bundle over a global field $k$ of characteristic not equal to $2$. Over local and global fields, the arithmetic of such conic bundles has been well studied. By the Hasse-Minkowski theorem, conic bundles over $\Q$ with two geometric singular fibers always satisfy the Hasse principle and work of Iskovskikh shows that conic bundles over $\Q$ with three geometric singular fibers always have rational points \cite{Isk96}. However, there exist conic bundles with 4 geometric singular fibers, specifically Ch\^atelet surfaces, that fail the Hasse principle, see \cite{Isk71} for the case over $\Q$ and \cite{Poo09}*{\S 5} for the general construction over a number field.

Since conics have numerous quadratic points, for any conic bundle there are numerous quadratic extensions over which they gain global points. However, we show that a much stronger statement holds, namely that for a large class of conic bundles with $4$ geometric singular fibers, the Hasse principle holds over \textbf{every} even degree extension. The main arithmetic idea in the proof dates back to work of Kanevsky \cite{Kan85}.

\begin{Theorem}\label{thm: main thm four fibers}
    Let $k$ be a number field, let $X \to \PP^1_k$ be a conic bundle with four geometric singular fibers, and assume that either $\frac{\Br(X)}{\Br_0(X)} \neq 0$ or $X(\A_k) \neq \emptyset $.  Then, if $L/k$ is an even degree extension, we have
    $$ X(L) \neq \emptyset \Leftrightarrow X(\A_L) \neq \emptyset. $$
\end{Theorem}

For arbitrary conic bundles we do not have such unconditional results, however, we can extend our prior results to conic bundles with more geometric singular fibers at the expense of assuming Schinzel's hypothesis as well as restricting to quadratic extensions.

\begin{Theorem}\label{thm: main thm arb fibers}
    Let $k$ be a number field and let $X \to \PP^1_k$ be a conic bundle.  Then, there exists an explicit finite set $S$ of quadratic extensions of $k$ such that for every quadratic extension $L/k \notin S$, we have
    $$ X(\A_L)^{\Br} \neq \emptyset \Leftrightarrow X(\A_L) \neq \emptyset. $$
    In particular, if $X$ has at most $5$ geometric singular fibers, or assuming Schinzel's hypothesis, then $X$ satisfies the Hasse principle over $L$.
\end{Theorem}

The set $S$ depends only on the closed points $P \in \PP^1_k$ over which $X$ has a singular fiber, together with the residues at such $P$ of the generic fiber $X_{\eta}$ considered as an element of $\Br(k(t))$.

The proofs of \Cref{thm: main thm four fibers} and \Cref{thm: main thm arb fibers} consist of two steps.  We first show that Brauer classes in the image of the restriction map $\Res_{L/k}\colon \frac{\Br(X)}{\Br_0(X)} \to \frac{\Br(X_L)}{\Br_0(X_L)}$ for even degree extensions $L/k$ do not obstruct the existence of rational points.  Then, we provide a characterization of when the restriction map is surjective (and therefore the Brauer-Manin set is nonempty) in terms of the singular fibers and the residues of the Brauer class corresponding to the generic fiber at these points.  Combining work of \CT and Coray \cite{CTC79}, \CT, Sansuc, and Swinnerton-Dyer \cite{CTSSD87}, and \CT and Swinnerton-Dyer \cite{CTSD94} show that the Brauer-Manin obstruction is the only obstruction to rational points for conic bundles with five or fewer geometric singular fibers, and if one assumes Schinzel's hypothesis, this can be extended to conic bundles with arbitrarily many geometric singular fibers, see \Cref{lem: BMO only obstruction}.

In 2009, Pete Clark coined the notion of a potential Hasse principle failure for smooth, geometrically irreducible varieties over number fields. In particular, given such a variety $V/k$, we say that it is a \emph{potential Hasse principle failure} (PHPF) if there exists an extension $L/k$ such that $V$ fails the Hasse principle over $L$. Clark showed the existence of infinitely many curves which were PHPFs. It was also conjectured that for any positive genus curve $C/k$ with no $k$-rational points, $C$ is a PHPF, see \cite{Cla05}*{Theorem 1, Conjecture 4}. \Cref{thm: main thm four fibers} shows that most conic bundles with four geometric singular fibers are not PHPFs, although there do exist cases where new Brauer classes can give a Brauer-Manin obstruction over $L$.  If the conditions of \Cref{thm: main thm four fibers} are not met, it is possible for there to exist a quadratic extension $L/k$, where the singular fibers of $X_L \rightarrow \PP_L^1$ lie over two closed points of degree $2$, over which the Hasse principle could possibly fail.

\begin{Theorem}\label{thm: PHPF}
    The Ch\^atelet surface $X/\Q$ given by
    $$ y^2 - 5z^2 = \frac{3}{5}(5t^4 + 7t^2 + 1) $$
    has no $\A_{\Q}$ points, but fails the Hasse principle over $L = \Q(\sqrt{29})$, i.e. $X$ is a potential Hasse principle failure.
\end{Theorem}

Indeed, we show that for Ch\^atelet surfaces, quadratic extensions not contained in the set $S$ described by \Cref{thm: main thm arb fibers} are the only even degree extensions over which the Hasse principle can fail, and for general conic bundles with four singular fibers, there are only two possible types of behavior (see \Cref{cor: avoid three quadratic exts}). In addition, we trace the failure of the Hasse principle to a parity condition on the number of ramification places of a conic that splits in a fixed quadratic extension (see \Cref{thm: phpf converse}).

\section*{Acknowledgements}

The authors would like to thank our advisor Bianca Viray for her constant support and many fruitful discussions that led to these results, and her assistance in constructing the example of \Cref{thm: PHPF}.  The authors also thank Masahiro Nakahara and Isabel Vogt for their ideas regarding this project. Additionaly, we thank Jean-Louis \CT for providing helpful suggestions regarding the statements of \Cref{thm: main thm four fibers} and \Cref{thm: main thm arb fibers} and Jen Berg, David Harari, Dan Loughran, and Olivier Wittenberg for their helpful comments.

\section{Preliminaries}

\subsection{Notation}

Throughout this paper, let $k$ be a global field of characteristic not equal to $2$.  We say that $X \to \PP^1_k$ is a conic bundle over $\PP^1_k$ if $X$ is smooth and geometrically integral over $k$ and the morphism to $\PP^1_k$ is proper where the generic fiber is a smooth conic.  We assume that $X$ is minimal; that is, there does not exist a $(-1)$-curve on $X$ which can be blown down.

\begin{Lemma}\label{lem: even degree conic has points}
Let $k$ be a local field and let $C/k$ be a smooth conic. If $L/k$ is an even degree extension then $C(L)\neq \emptyset$.
\end{Lemma}
\begin{proof}
    From local class field theory, the restriction map factors as 
    $$\Res_{L/k}\colon \Br k \cong \Q/\Z \xrightarrow{[L : k]} \Q/\Z \cong \Br L$$
    Using the correspondence between smooth conics and quaternion algebras, we can apply this to any element of $\Br k[2]$ and obtain the result.
\end{proof}

\subsection{The Brauer group of a conic bundle}

After a change of coordinates, since $k$ is infinite, we may assume that the fiber over the point at infinity, $X_{\infty}$, is a smooth conic.  Let $S_k$ be the finite set of closed points in $\A^1_k$ with geometric singular fiber, where we will write $S = S_k$ when the ground field is unspecified, and let $|S|$ denote the number of such closed points, not counting degree, i.e. if $X$ has four geometric singular fibers and $|S| = 1$, then $S$ contains one closed point of degree $4$.

For any point $P \in \A_{k}^1$, let $\mathbf{k}(P)$ denote its residue field.  For $P \in S$, the fiber $X_P$ degenerates to the union of two lines, $\ell_P$ and $\ell'_P$, defined over $\mathbf{k}(P)(\sqrt{\alpha_P})$ for some $\alpha_P \in \mathbf{k}(P)^*$.  Since $X$ is assumed to be minimal, if $P \in S$ is a degree $1$ point, then $\alpha_P$ cannot be a square in $k$, as both $\ell$ and $\ell'$ would be $(-1)$-curves defined over $k$.  Note that the generic fiber of $X \to \PP^1$ is a Severi-Brauer variety over $k(t)$ and therefore an element of $\Br(k(t))$, and $\alpha_P$ is the residue of this Brauer class at $P$.  The group $\frac{\Br(X)}{\Br_0(X)}$ is then determined by $\mathbf{k}(P)$ and $\alpha_P$ for $P \in S$, as shown in the following lemma.

\begin{Lemma}\label{lem: product of norms classifies brauer class}\cite{CTS21}*{Corollary 11.3.5}
    Let $\pi\colon X \to \PP^1_k$ be a conic bundle with singular fibers over $S \subseteq \PP^1$.  Let $V \subseteq \F_2^{|S|}$ be defined by
    $$ V = \left\{\ve = (\ve_P)_{P \in S} \in \F_2^{|S|} \,\middle|\, \prod_{P \in S} \Nm_{\mathbf{k}(P)/k}(\alpha_P)^{\ve_P} \in k^{\times 2} \right\}. $$
    Then, $\frac{\Br(X)}{\Br_0(X)}$ is isomorphic to $V$ modulo the subspace generated by $(1,\dots,1)$.  The vector $\ve \in V$ corresponds to the image of $\pi^*(A_{\ve}) \in \Br(X)$, defined by
    $$ A_{\ve} = \sum_{P \in S} \ve_P \Cor_{\mathbf{k}(P)(t)/k(t)}\calA_P $$
    where $\calA_P = (\alpha_P , t - \tau_P) \in \Br(\mathbf{k}(P)(t))$ and $\tau_P \in \mathbf{k}(P)$ denotes the image of $t$ in $\mathbf{k}(P) = k[t]/P(t)$.
\end{Lemma}

In particular, in order for this quotient to be well-defined, note that $\prod_{P \in S}\Nm_{\mathbf{k}(P)/k}(\alpha_P) \in k^{\times 2}$ \cite{Sko15}*{Lemma 2.2}.  For a detailed explanation of this construction, see \cite{CTS21}*{Section 11.3.1}.

\begin{Corollary}\label{lem: trivial classes at infinity generate brauer}
   Let $\pi\colon X \to \PP^1_k$ be a conic bundle.  We have that $\frac{\Br(X)}{\Br_0(X)}$ is generated by images of elements of $\ker (\Br(X)[2] \rightarrow \Br(X_{\infty}))$.
\end{Corollary}
 
\begin{proof}
    Use the notation of \Cref{lem: product of norms classifies brauer class}.  For each $P \in S$, we have that $\mathcal{A}_P(\infty)=(1,a_P)=0 \in \Br(\mathbf{k}(P))$. Since $\Cor_{\mathbf{k}(P)/k}$ is a group homomorphism, we have that each
    $A_{\ve} \in \ker (\Br(\mathbf{k}(\PP^1))[2] \xrightarrow{\infty^*} \Br(k)[2])$ and so $\pi^*(A_{\varepsilon}) \in \ker (\Br(X)[2] \rightarrow \Br(X_{\infty}))$.
\end{proof}

\begin{Lemma}\label{lem: unchanged S and norms make res isom}
    Let $L/k$ be a finite extension such that $S_L = S_k$.  For every $T \subseteq S$, let $\beta_T := \prod_{P \in T} \Nm_{\mathbf{k}(P)/k}(\alpha_P)$. 
 If $\beta_T \in k^{\times 2} \Leftrightarrow \beta_T \in L^{\times 2}$ for all $T \subseteq S$, then the map $\Res_{L/k}\colon \frac{\Br(X)}{\Br_0(X)} \to \frac{\Br(X_L)}{\Br_0(X_L)}$ is an isomorphism.
\end{Lemma}

\begin{proof}
    Using \cite{CTS21}*{Proposition 11.3.4} on both $k$ and $L$, we have the diagram
    $$
    \begin{tikzcd}
    0 \arrow[r] & \Br(k) \arrow[r] \arrow[d] & \Br(X) \arrow[r] \arrow[d] & {\F_2^{|S_k|}/(1,\dots,1)} \arrow[r] \arrow[d] & k^{\times}/k^{\times 2} \\
    0 \arrow[r] & \Br(L) \arrow[r]           & \Br(X_L) \arrow[r]         & {\F_2^{|S_L|}/(1,\dots,1)} \arrow[r]         & L^{\times}/L^{\times 2}
    \end{tikzcd}
    $$
    where the last horizontal map sends $\ve \in \F_2^{|S_k|}/(1,\dots,1)$ to $\prod_{P \in S} \Nm_{\mathbf{k}(P)/k}(\alpha_P)^{\ve_P}$ and analogously for $L$, and the vertical maps are induced by $\Res_{L/k}$.  In particular, we have that the map $\F_2^{|S_k|}/(1,\dots,1) \to \F_2^{|S_L|}/(1,\dots,1)$ sends $e_P$, the basis vector corresponding to $P$, to the sum $e_{Q_1} + \cdots + e_{Q_n}$, where $P$ splits into $Q_1,\dots,Q_n$ over $L$.  The diagram above induces isomorphisms compatible with restriction
    $$
    \begin{tikzcd}
    \frac{\Br(X)}{\Br(k)} \arrow[r] \arrow[d] & {\ker\left( \F_2^{|S_k|}/(1,\dots,1) \to k^{\times}/k^{\times 2}\right)} \arrow[d] \\
    \frac{\Br(X_L)}{\Br(L)} \arrow[r]         & {\ker\left( \F_2^{|S_L|}/(1,\dots,1) \to L^{\times}/L^{\times 2}\right)}        
    \end{tikzcd}
    $$
    where the vertical maps are given by $\Res_{L/k}$.  Since $S = S_L$ and using the explicit description of the map above, the map $\F_2^{|S_k|}/(1,\dots,1) \to \F_2^{|S_L|}/(1,\dots,1)$ is an isomorphism.  The condition on $\beta_T$ precisely implies that the kernels will be equal, and therefore $\Res_{L/k}$ is an isomorphism as desired.
\end{proof}

\begin{Lemma}\label{cor: no rat pts and brauer nonzero case}
    Let $k$ be a global field of characteristic not equal to $2$ and let $X \to \PP^1_k$ be a conic bundle with four geometric singular fibers.  If $X(k) = \emptyset$ and $\frac{\Br(X)}{\Br_0(X)} \neq 0$, then $S = \{P_1,P_2\}$ with $\deg(P_i) = 2$ and $\Nm_{\mathbf{k}(P_i)/k}(\alpha_{P_i}) \in k^{\times 2}$ for $i \in \{1,2\}$, and $\frac{\Br(X)}{\Br_0(X)} = \Z/2\Z$.
\end{Lemma}
\begin{proof}
    Using the description of $\frac{\Br(X)}{\Br_0(X)}$ in \Cref{lem: product of norms classifies brauer class}, we proceed by cases on the size of $S$ and degrees of points in $S$.  If $|S| = 1$, then $V \cong \F_2$ giving a quotient of $0$, which violates our assumption that $\frac{\Br(X)}{\Br_0(X)} \neq 0$.
    If any of the singular fibers of $X$ lie over a degree $1$ point $P\in \A^1_k$, then we can obtain a $k$-rational point on $X$ by taking the intersection point of the lines $\ell_P$ and $\ell'_P$ for such $P \in S$.  Thus, if $|S| = 3$ or $|S| = 4$, then at least one $P \in S$ must be degree $1$, violating the assumption that $X(k) = \emptyset$.
    Therefore, if $X(k) = \emptyset$ and $\frac{\Br(X)}{\Br_0(X)} \neq 0$, we must have $|S| = 2$ consisting of two points of degree $2$.  If either $\Nm_{\mathbf{k}(P_i)/k}(\alpha_{P_i}) \notin k^{\times 2}$, then $\frac{\Br(X)}{\Br_0(X)} = 0$, so both must be non-squares in $k$.
\end{proof}

\begin{Corollary}\label{cor: not surjective case}
    Let $k$ be a global field of characteristic not equal to $2$ and let $X \to \PP^1_k$ be a conic bundle with four geometric singular fibers.  If there exists an even degree extension $L/k$ such that $X(L) = \emptyset$ and $\Res_{L/k}\colon \frac{\Br(X)}{\Br_0(X)} \to \frac{\Br(X_L)}{\Br_0(X_L)}$ is not surjective, then either
    \begin{enumerate}[label=$(\roman*)$]
        \item\label{case: Sk = 1 , SL = 2} $|S_k| = 1$ and $|S_L| = 2$ with $P,Q \in S_L$ both degree $2$ and interchanged by $\Gal(L/k)$, or
        \item\label{case: Sk = 2 , SL = 2} $|S_k| = |S_L| = 2$ with $P,Q \in S_L$ both degree $2$ and fixed by $\Gal(L/k)$, with $\Nm_{\mathbf{k}(P)/k}(\alpha_P) \notin k^{\times 2}$ and $\Nm_{L\cdot\mathbf{k}(P)/L}(\alpha_P) \in L^{\times 2}$ (equivalently for $Q$).
    \end{enumerate}
    Furthermore, in both cases, we must have $\frac{\Br(X)}{\Br_0(X)} = 0$.
\end{Corollary}
\begin{proof}
    Let $L/k$ be an even degree extension such that $X(L) = \emptyset$ and $\Res_{L/k}\colon \frac{\Br(X)}{\Br_0(X)} \to \frac{\Br(X_L)}{\Br_0(X_L)}$ is not surjective.  In particular, $\frac{\Br(X_L)}{\Br_0(X_L)} \neq 0$, so by \Cref{cor: no rat pts and brauer nonzero case}, we must have that $S_L = \{P,Q\}$ both of degree $2$, $\frac{\Br(X_L)}{\Br_0(X_L)} \cong \Z/2\Z$ and $\Nm_{L\cdot\mathbf{k}(P)/L}(\alpha_P) \in L^{\times 2}$.  Since $S_L$ is the base change of $S_k$, either $S_k$ consists of one point of degree $4$ split into $P$ and $Q$ over $L$, which gives case \ref{case: Sk = 1 , SL = 2}.  Otherwise, $S_k$ consists of $P$ and $Q$ both degree $2$ points over $k$.  If $\Nm_{\mathbf{k}(P)/k}(\alpha_P) \in k^{\times 2}$, the hypotheses of \Cref{lem: unchanged S and norms make res isom} would be satisfied, and thus the restriction map would be an isomorphism.  We conclude that $\Nm_{\mathbf{k}(P)/k}(\alpha_P) \notin k^{\times 2}$, which gives case \ref{case: Sk = 2 , SL = 2}.  Additionally, since $X(k) \subseteq X(L) = \emptyset$, the contrapositive of \Cref{cor: no rat pts and brauer nonzero case} shows that $\frac{\Br(X)}{\Br_0(X)} = 0$.
\end{proof}

\begin{Remark}
    If $X$ is a Ch\^atelet surface given by the affine equation $y^2 - az^2 = f(x)$ for a quartic polynomial $f(x)$, then $S$ is geometrically given by the roots of $f$ over a splitting field, $\alpha_P = a$ for all $P \in S$, and the only possible case of \Cref{cor: not surjective case} is \ref{case: Sk = 1 , SL = 2}.
\end{Remark}

The following theorem is a consequence of work of \CT and Coray, \CT and Swinnerton-Dyer, and \CT, Sansuc, and Swinnerton-Dyer, and is a critical ingredient in our results, as it implies that in order to prove the existence of rational points, it suffices to show that the Brauer-Manin set is nonempty.

\begin{Theorem}\label{lem: BMO only obstruction}
    Let $X/k$ be a conic bundle with five or fewer geometric singular fibers.  Then, $X(k) \neq \emptyset$ if and only if $X(\A_k)^{\Br(X_k)} \neq \emptyset$.  Moreover, if one assumes Schinzel's hypothesis, the statement is true for conic bundles with arbitrarily many geometric singular fibers.
\end{Theorem}
\begin{proof}
    For Ch\^atelet surfaces, this is proven in \cite{CTSSD87}*{Theorem B.i.a}.  Now assume that $X/k$ is a general a conic bundle with five or fewer geometric singular fibers.  Any point of $X(\A_k)^{\Br(X_k)}$ is an adelic $0$-cycle of degree $1$ orthogonal to the Brauer group, and 
    if there is no Brauer-Manin obstruction to the existence of a $0$-cycle of degree $1$, then there exists a $0$-cycle of degree $1$ \cite{CTSD94}*{Theorem 5.1}.  Then, $X(k) \neq \emptyset$ if and only if there exists a zero cycle of degree one over $k$ \cite{CTC79}*{Corollaire 4}.  The case of arbitrarily many geometric singular fibers assuming Schinzel's hypothesis is proven in \cite{CTSD94}*{Theorem 4.2}.
\end{proof}

\section{Proofs of main theorems}

\begin{Theorem}\label{thm: restricted brauer classes nonempty}
Let $k$ be a global field of characteristic not equal to $2$ and let $X \rightarrow \PP^1_k$ be a conic bundle with four geometric singular fibers.  Then, for any even degree extension $L/k$ $$X(\A_L)^{\Res_{L/k}(\Br(X_k))}\neq \emptyset \Leftrightarrow X(\A_L)\neq \emptyset.$$ In particular, if $\Res_{L/k}\colon \frac{\Br(X)}{\Br_0(X)} \rightarrow \frac{\Br(X_L)}{\Br_0(X_L)}$ is surjective, then $X(\A_L)^{\Br (X_L)}\neq \emptyset \Leftrightarrow X(\A_L)\neq \emptyset$.
\end{Theorem}

\begin{proof}
    We begin by observing that one direction is immediate since $X(\mathbb{A}_L)=\emptyset$ implies that $X(\A_L)^{\Res_{L/k}(\Br(X_k))} = \emptyset$. Now assume that $X(\A_L)\neq \emptyset$ and let $\Omega_k$ denote the set of places of $k$.
    
    We first show that for all $v \in \Omega_k$ and all $w\mid v$, there exists a point $(P_w) \in X(L \otimes_k k_v)$ such that $\sum_{w\mid v} \inv_w( \ev_{\mathcal{A}}(P_w))=0$ for all $\mathcal{A}\in \ker (\Br(X)[2] \rightarrow \Br(X_{\infty}))$.    We first consider the case when $X(k_v)\neq \emptyset$. Choose a point $P_v\in X(k_v)$, and for all $w\mid v$, set $P_w=P_v$. Then 
    \begin{align*}
        \sum_{w\mid v} \inv_w(\ev_{\mathcal{A}}(P_w)) &=\sum_{w\mid v} \inv_w(\Res_{L_w/k_v}(\ev_{\mathcal{A}}(P_v)))=\sum_{w\mid v}[L_w: k_v]\inv_v(\ev_{\mathcal{A}}(P_v)) \\
        &=\inv_v(\ev_{\mathcal{A}}(P_v))\sum_{w\mid v}[L_w: k_v]=\inv_v(\ev_{\mathcal{A}}(P_v))\cdot [L: k]=0\in \Q/\Z.
    \end{align*}
    Now consider the case where $X(k_v) = \emptyset$.  By \cite{CTC79}, $X(k_v)\neq \emptyset$ if and only if there exists a 0-cycle of degree $1$, hence $X(k_v) = \emptyset$ implies that $X(L_w) = \emptyset$ for all odd degree extensions $L_w/k_v$.  Since $X(L \otimes_k k_v) \neq \emptyset$, $L_{w}/k_v$ is an even degree extension for all $w\mid v$. Thus, by Lemma \ref{lem: even degree conic has points}, there exists a point $P_w \in X_{\infty}(L_w)$, and this satisfies $\inv_w( \ev_{\mathcal{A}}(P_w))=0$ as $\mathcal{A}\in \ker (\Br(X)[2] \rightarrow \Br(X_{\infty}))$.
    
    Following this construction for each place $v \in \Omega_k$, we produce $(P_w) \in X(\A_L)$ satisfying 
    
    $$\sum_{v \in \Omega_k} \sum_{w\mid v} \inv_w( \ev_{\mathcal{A}}(P_w))=0.$$ 

    By Corollary \ref{lem: trivial classes at infinity generate brauer}, we have that $\frac{\Br(X)}{\Br_0(X)}$ is generated by images of elements of $\ker (\Br(X)[2] \rightarrow \Br(X_{\infty}))$, and therefore we have that $X(\A_L)^{\Res_{L/k}(\Br(X_k))}\neq \emptyset$.
\end{proof}

\begin{proof}[Proof of \Cref{thm: main thm four fibers}]
    By work of \CT and Coray, \CT and Swinnerton-Dyer, and \CT, Sansuc, and Swinnerton-Dyer (see \Cref{lem: BMO only obstruction}), we have that $X(k) \neq \emptyset$ if and only if $X(\A_k)^{\Br(X)} \neq \emptyset$, so it suffices to show the Brauer-Manin set is nonempty.

    Suppose $\frac{\Br(X)}{\Br_0(X)} = 0$ and $X(\A_k) \neq \emptyset$.  Then, we have that
    $$ \emptyset \neq X(\A_k) = X(\A_k)^{\Br(X)} $$
    and therefore $X(k) \neq \emptyset$ and so $X(L) \neq \emptyset$ for every even degree $L/k$.

    Suppose $\frac{\Br(X)}{\Br_0(X)} \neq 0$.  If $X(L) \neq \emptyset$, we are done, so assume $X(L) = \emptyset$.  By \Cref{thm: restricted brauer classes nonempty}, it suffices to show that $\Res_{L/k}\colon \frac{\Br(X)}{\Br_0(X)} \to \frac{\Br(X_L)}{\Br_0(X_L)}$ is surjective.  If $\Res_{L/k}\colon \frac{\Br(X)}{\Br_0(X)} \to \frac{\Br(X_L)}{\Br_0(X_L)}$ is not surjective, \Cref{cor: not surjective case} shows $\frac{\Br(X)}{\Br_0(X)} = 0$, giving a contradiction.
\end{proof}

\begin{Corollary}\label{cor: avoid three quadratic exts}
Let $k$ be a global field of characteristic not equal to $2$, let $X\rightarrow \PP^1_k$ be a conic bundle with four geometric singular fibers. There exists a minimal collection (of size at most three) of quadratic extensions $k_i/k$ such that for all even degree extensions $L/k$ that do not contain any $k_i$, we have 
$$X(\A_L)^{\Br (X_L)}\neq \emptyset \Leftrightarrow X(\A_L)\neq \emptyset.$$
\end{Corollary}

\begin{proof}
By \Cref{thm: restricted brauer classes nonempty}, it suffices to show that if $L/k$ is an even degree extension such that $\Res_{L/k}\colon \frac{\Br(X)}{\Br_0(X)} \rightarrow \frac{\Br(X_L)}{\Br_0(X_L)}$ is not surjective, then $L$ contains one of at most three possible quadratic extensions $k_i/k$.  By \Cref{cor: not surjective case}, if the restriction map is not surjective, then there are two cases to consider.

Suppose $|S_k| = 1$ and $|S_L| = 2$, where the two points of $L$ are degree $2$ and interchanged by $\Gal(L/k)$.  Here, the problematic extensions $k_i/k$ are precisely the extensions for which the singular fibers of $X_{k_i} \to \PP^1_{k_i}$ lie over two degree $2$ closed points that are interchanged by $\Gal(k_i/k)$.  Then, the unique closed point of $S_k$ is of degree $4$ and corresponds to either a polynomial of the form $\Nm_{k_0/k}(g(t))$ for $k_0/k$ quadratic and $g(t) \in k_0[t]$ an irreducible quadratic, or of the form $\Nm_{F/k}(\ell(t))$ for $F/k$ biquadratic and $\ell(t) \in F[t]$ linear.  Thus, $L$ will contain $k_0$ in the first case or any of the three quadratic subfields of $F$ in the second case.

Suppose $|S_k| = 2$ and $|S_L| = 2$ with $\Nm_{\mathbf{k}(P)/k}(\alpha_P) \notin k^{\times 2}$ and $\Nm_{L\cdot\mathbf{k}(P)/L}(\alpha_P) \in L^{\times 2}$ for $P \in S_k = S_L$.  Since the points of $S_k$ do not split over $L$, we must have that $\mathbf{k}(P)$ is linearly disjoint from $L$, and therefore $\Nm_{L\cdot\mathbf{k}(P)/L}(\alpha_P) = \Nm_{\mathbf{k}(P)/k}(\alpha_P)$.  In order for $\Nm_{L\cdot\mathbf{k}(P)/L}(\alpha_P)$ to be a square in $L$, the field $L$ must then contain the quadratic field $k(\sqrt{\Nm_{\mathbf{k}(P)/k}(\alpha_P)})$.
\end{proof}

When $X$ has arbitrarily many singular fibers, we can extend \Cref{thm: restricted brauer classes nonempty} at the expense of restricting to quadratic extensions.

\begin{Theorem}\label{thm: restricted brauer arbitrary bad fibers}
Let $k$ be a global field of characteristic not equal to $2$ and let $X \rightarrow \PP^1_k$ be a conic bundle, then for any quadratic extension $L/k$ $$X(\A_L)^{\Res_{L/k}(\Br(X_k))}\neq \emptyset \Leftrightarrow X(\A_L)\neq \emptyset.$$
\end{Theorem}

\begin{proof}
Observe that one direction is again immediate, hence we assume that $X(\A_L)\neq \emptyset$. In a similar fashion to the proof of \Cref{thm: restricted brauer classes nonempty}, we will show that for all $v \in \Omega_k$ there exists a point $(P_w) \in X(L \otimes_k k_v)$ such that $\sum_{w|v} \inv_w( \ev_{\mathcal{A}}(P_w))=0$ for all $\mathcal{A}\in \ker (\Br(X)[2] \rightarrow \Br(X_{\infty}))$, and thus $X(\A_L)^{\Res_{L/k}(\Br(X_k))}\neq \emptyset$. 

If $X(k_v) \neq \emptyset$ then by an identical argument to the proof of \Cref{thm: restricted brauer classes nonempty}, there exists a point $(P_w) \in X(\A_L)$ such that $\sum_{w|v} \inv_w(\ev_{\mathcal{A}}(P_w))=0$ for all $\mathcal{A}\in \ker (\Br X[2] \rightarrow \Br X_{\infty})$.

If $X(k_v) = \emptyset$ and $X(L \otimes_k k_v)\neq \emptyset$, then $v$ is non-split and so $L_w/k_v$ is quadratic. Thus, \Cref{lem: even degree conic has points} implies the existence of a point $P_w \in X_{\infty}(L_w)$ which satisfies $\inv_w( \ev_{\mathcal{A}}(P_w))=0$ for all $\mathcal{A}\in \ker (\Br(X)[2] \rightarrow \Br(X_{\infty}))$.
\end{proof}

When considering extensions analogous to the extensions $k_i$ of \Cref{cor: avoid three quadratic exts} for arbitrary conic bundles there are more issues that can arise, but nonetheless, there are still only finitely many. These problematic extensions arise in a similar manner to those of \Cref{cor: avoid three quadratic exts}. Namely, if $S\subseteq \PP^1_k$ is the finite set of closed points with geometric singular fiber, then the problematic extensions, $k_i/k$, are those which either coincide with residue fields $\mathbf{k}(P)$ for $P \in S$, or adjoin square roots of norms of residues $\Nm_{\mathbf{k}(P)/k}(\alpha_P)$.  Over these extensions, $\Res_{k_i/k} \colon \frac{\Br(X)}{\Br_0(X)}\rightarrow \frac{\Br(X_{k_i})}{\Br_0(X_{k_i})}$ could fail to be surjective and new Brauer classes in $\frac{\Br(X_{k_i})}{\Br_0(X_{k_i})}$ could give a Brauer-Manin obstruction to $X_{k_i}$. For all extensions $L/k$ over which both the points in $S$ remain unchanged and the (products of) norms of residues do not become squares, \Cref{lem: unchanged S and norms make res isom} tells us that $\Res_{L/k}$ is an isomorphism, and indeed $X(\A_L)^{\Br (X_L)} \neq \emptyset \Leftrightarrow X(\A_L)\neq \emptyset$. Since there are only finitely many geometric singular fibers, there are only finitely many problematic extensions of this form, hence we have the following corollary.

\begin{Corollary}\label{cor: arbitrary bad fibers avoid finitely many fields}
Let $k$ be a global field of characteristic not equal to $2$, let $X \rightarrow \PP^1_k$ be a conic bundle, and let $L/k$ be a quadratic extension.  For every $T \subseteq S$, let $\beta_T := \prod_{P \in T} \Nm_{\mathbf{k}(P)/k}(\alpha_P)$.  Define
$$M = \left(\bigcup_{P \in S} \{\mathbf{k}(P)\}\right) \cup \left( \bigcup_{T \subseteq S} \left\{k\left(\sqrt{\beta_T}\right)\right\}\right).$$
Then, if $L \not\subseteq k_i$ for all $k_i \in M$, then $$X(\A_L)^{\Br (X_L)}\neq \emptyset \Leftrightarrow X(\A_L)\neq \emptyset.$$
\end{Corollary}

\begin{proof}
    Since $L$ is disjoint from the fields contained in $M$, we can apply \Cref{lem: unchanged S and norms make res isom} to show that $\Res_{L/k} \colon \frac{\Br(X)}{\Br_0(X)}\rightarrow \frac{\Br(X_{L})}{\Br_0(X_L)}$ is an isomorphism.  The result then follows using \Cref{thm: restricted brauer arbitrary bad fibers}.
\end{proof}

\begin{proof}[Proof of \Cref{thm: main thm arb fibers}]
    This follows from \Cref{cor: arbitrary bad fibers avoid finitely many fields} and the case of \Cref{lem: BMO only obstruction} using Schinzel's Hypothesis.
\end{proof}

\begin{Remark}
\emph{Quite generally, if $k$ is a number field and $X/k$ is a smooth, projective, and geometrically rational variety, then Pic$(\overline{X})$ is a Galois lattice split by a finite Galois extension $K/k$. If $E/k$ is a finite field extension which is linearly disjoint from $K/k$, then the map $\Res_{E/k}\colon \frac{\Br(X)}{\Br_0(X)} \rightarrow \frac{\Br(X_E)}{\Br_0(X_E)}$ is an isomorphism, and in particular, is surjective. \Cref{cor: arbitrary bad fibers avoid finitely many fields} is a slight refinement of this result for the case of conic bundles, where the problematic fields are quadratic and made explicit.}
\end{Remark}

\begin{proof}[Proof of \Cref{thm: PHPF}]
    We first show that $X(\Q_p) \neq \emptyset$ if and only if $p = 3$.  Note that the fiber at infinity $X_{\infty}$ is given by $y^2 - 5z^2 = 3w^2$ which has $\Q_p$ points for all $p \neq 3,5$, and the fiber $X_1$ over $t = 1$ is given by $y^2 - 5z^2 = \frac{39}{5}$ which has $\Q_5$ points since $-39 \in \Q_5^{\times 2}$.  When $p = 3$, note that $5 \notin \Q_3^{\times 2}$, so $\Q_3(\sqrt{5})$ is a quadratic unramified extension of $\Q_3$, and $y^2 - 5z^2$ is a norm from $\Q_3(\sqrt{5})$ and thus has even valuation.  Since $5t^4 + 7t^2 + 1$ is irreducible in $\F_3$, the minimum valuation of each term is attained uniquely and thus always has even valuation, and so $\frac{3}{5}(5t^4 + 7t^2 + 1)$ always has odd valuation and cannot be of the form $y^2 - 5z^2$.  This shows that $X(\A_{\Q}) = \emptyset$.  Thus, $X(\Q) = \emptyset$, and since $5t^4 + 7t^2 + 1$ is irreducible, by \Cref{cor: no rat pts and brauer nonzero case} we must have $\frac{\Br(X)}{\Br_0(X)} = 0$.  

    Now consider $L = \Q(\sqrt{29})$.  Since $[\Q_3(\sqrt{29}):\Q_3]=2$, \Cref{lem: even degree conic has points} shows $X(\Q_3(\sqrt{29})) \neq \emptyset$ and thus $X(\A_L) \neq \emptyset$.
    We see that $X_L$ is given by 
    
    $$y^2-5z^2=3\left(t^2+\frac{1}{10}(7+\sqrt{29})\right)\left(t^2+\frac{1}{10}(7-\sqrt{29})\right)$$
    
    \noindent hence by a computation using \Cref{lem: product of norms classifies brauer class} we have $\frac{\Br(X_L)}{\Br_0(X_L)}\cong \Z/2\Z=\langle \calA \rangle$, where $\calA$ denotes the quaternion algebra $\left(5,t^2+\frac{1}{10}(7+\sqrt{29})\right)$.

    It remains to show that $X(\A_L)^{\Br(X_L)} = \emptyset$ and we do so by first showing that $\ev_{\calA} \colon X(L_w) \rightarrow  \Br(L_w)$ is identically zero for all $w \nmid 5$. Begin by observing that for all primes $w\mid p$ where $p\neq 3,5$, we know that $X_{\infty}(L_w)\neq \emptyset$, hence $\ev_{\calA} \colon X(L_w) \rightarrow  \Br(L_w)$ takes the value $0$ at such primes, as $\calA$ evaluates to $(5,1) = 0 \in \Br(L_w)$. Since the evaluation map is constant at all primes $w$ of good reduction \cite{CTS13}*{Theorem 3.1}, it remains to check $\ev_{\mathcal{A}}\left(X(L_w)\right)$ for primes $w$ lying over $p=2,3,5,$ and $29$. Now observe that, for $w\mid 2,3,29$, we have that $5 \in L_w^{\times 2}$ hence the algebra $\mathcal{A}=\left(5,t^2+\frac{1}{10}(7+\sqrt{29})\right)$ is identically zero and $\ev_{\mathcal{A}}\left(X(L_w)\right)=0$. Thus for any $(P_w)\in X(\A_L)$ we have 

    $$\sum_{w \in \Omega_L} \inv_w\left(\ev_{\mathcal{A}}(P_w)\right)=\sum_{w\mid 5}\inv_w \left(\ev_{\calA}(P_w)\right).$$

    Let $w_1$ and $w_2$ denote the places lying over $5$ corresponding to the embeddings in which $\sqrt{29} \equiv 2 \pmod 5$ and $\sqrt{29} \equiv 3 \pmod 5$ respectively and let $P_i=(t_i,y_i,z_i) \in X(L_{w_i})$. We now show that 

    $$
    \inv_{w_i}\left(\mathcal{A}(P_i)\right) =
    \begin{cases} 
      \frac{1}{2} & \text{if } i=1 \\
      0 & \text{if } i=2 
   \end{cases}.
    $$

    Let $\alpha=\frac{1}{10}(7+\sqrt{29})$ and $\overline{\alpha}=\frac{1}{10}(7-\sqrt{29})$. Over $\Q_5$, we have 
    $$X/\Q_{5}\colon y^2-5z^2=3(t^2+\alpha)(t^2+\overline{\alpha})$$
    hence $\frac{\Br X_{\Q_{5}}}{{\Br \Q_{5}}}$ is generated by the quaternion algebra $\mathcal{A}=(5,t^2+\alpha)$. Note that in $\Br X$, the quaternion algebra $\mathcal{A}$ is equivalent to the algebra $\mathcal{B}=(5,3(t^2+\overline{\alpha}))$. This will be an important tool in the final part of our proof.

First we consider the place $w_1$ and observe that since $\sqrt{29}\equiv 2\pmod 5$, we have $w_1(\alpha)=-1$. Moreover, since $\alpha\overline{\alpha}=\frac{1}{5}$, it follows that $w_1(\overline{\alpha})=0$. Take $P_1=(t_1,y_1,z_1)$ and suppose $w_1(t_1)<0$, and consider the polynomial $P(t)=\frac{3}{5}(5t^4+7t^2+1)$. We have that $5^{-4w_1(t_1)}P(t)\equiv 3 \pmod 5$ hence $P(t)$ is never a norm from $\Q_5(\sqrt{5})$, so we must have $w_1(t_1)\geq 0$ for all $P_1 \in X(L_{w_1})$.  Then, by the strong triangle inequality, $w_1(t_1^2+\alpha)=w_1(\alpha)$ and since $5\alpha\equiv \frac{1}{2}(7+\sqrt{29})\equiv 2 \pmod 5$, it follows that $10(t_1^2+\alpha) \in \Q_5^{\times 2}$, meaning $\ev_{\mathcal{A}}(P_1)=(5,10)$ which is a non-split quaternion algebra in $\Br(\Q_5)$. This shows that $\inv_{w_1}\left(\ev_{\mathcal{A}}(P_1)\right)=\frac{1}{2}$.

We now consider the case of $w_2$ and recall that in this case, $\sqrt{29}\equiv 3 \pmod 5$, hence $w_2(\alpha)=0$ and $w_2(\overline{\alpha})=-1$.  Take $P_2=(t_2,y_2,z_2)$ and observe that just as in the previous case, if $P_2 \in X(L_{w_2})$ we must have $w_2(t_2)\geq 0$. If $w_2(t_2)>0$ then by the same argument for when $w_1(t_1)\geq 0$, we have that $t_2^2+\alpha$ is always a square in $\Q_5$, hence $\ev_{\mathcal{A}}(P_2)=0 \in \Br(\Q_5)$. It remains to consider the case when $w_2(t_2)=0$ and for this, we consider the algebra $\mathcal{B}=(5,3(t^2+\overline{\alpha}))$ and show that $\ev_{\mathcal{B}}(P_2)$ is identically zero in $\Br(\Q_5)$. For such values of $t_2$ we have $w_2\left(3(t^2+\overline{\alpha})\right)=w_2(\overline{\alpha})$ hence $5\left(3(t^2+\overline{\alpha}\right)) \equiv 3\overline{\alpha} \pmod 5$, which by Hensel's lemma, is a square in $\Q_5$. It now follows that for all $P_2\in X(L_{w_2})$ we have $\inv_{w_2}(\ev_{\mathcal{A}}(P_2))=0.$

We can now conclude that for any $P_w \in X(\A_{L})$ we have $$\sum_{w\in \Omega_{L}} \inv_w\left(\ev_{\mathcal{A}}(P_w)\right)=\frac{1}{2}$$ and thus $X$ fails the Hasse principle over $L$.
\end{proof}

\section{An additional condition on Ch\^atelet surfaces}

We have showed that the failure of the Hasse principle for $X$ over an even degree extension is explained by quadratic extensions associated to the singular fibers of $X$.  In the case of Ch\^atelet surfaces, this can occur for surfaces $X/k$ of the form 
$$y^2-az^2=c\Nm_{F/k}(g(t))$$
where $a,c \notin k^{\times 2}$, $g(t) \in F[t]$ is a monic, irreducible polynomial, and $\Nm_{F/k}$ denotes the usual norm map. In contrast to \Cref{thm: PHPF}, not all such extensions $F/k$ are guaranteed to produce a Brauer-Manin obstruction over $F$. In order for such an obstruction to exist, the fiber $X_\infty$, the places for which it has no local points, and the places of bad reduction, must satisfy several necessary conditions.
\begin{Theorem}\label{thm: phpf converse}
Let $k$ be a number field, let $F/k$ be a quadratic extension, and let $X/k$ be the Ch\^atelet surface given by 
$$y^2-az^2=c\Nm_{F/k}(g(t))$$
where $a,c \notin k^{\times 2}$ and $g(t) \in F[t]$ is monic irreducible of degree $2$. Let $\Omega_k$ denote the set of places of $k$, let $\Omega_F$ denote the set of places of $F$, and assume $X(\A_k)= \emptyset$.  If 
$$\sum_{\substack{v\in \Omega_{k} \\ v \; \text{splits in} \; F}} \inv_v(a,c)=0\in \Q/\Z$$
then for all even degree extensions $L/k$, we have $X(L)\neq \emptyset \Leftrightarrow X(\A_L)\neq \emptyset$.
\end{Theorem}

\begin{proof}
    Let $v \in \Omega_{k}$, let $w \in \Omega_{F}$ such that $w\mid v$. Let $\mathcal{A}=(a,g(t))$ be a generator of $\frac{\Br(X_F)}{\Br_0(X_F)}$ and let $\sigma$ denote the generator of $\Gal(F/k)$. If $[F_w: k_v]=2$ then there exists a point $P_w\in X_{\infty}(F_w)$ such that $\inv_w(\mathcal{A}(P_w))=0$ by \Cref{lem: even degree conic has points}. Now assume that $F_w=k_v$, then there exists a unique place $w'$ such that $w'\neq w$ and $w'\mid v$, hence $F_w=F_{w'}=k_v$. Take $P_v \in X(k_v)$ and set $P_v=P_w=P_{w'}=(t_v,y_v,z_v)$. We then have 
    \begin{align*}
    \inv_w(\mathcal{A}(P_v))+\inv_{w'}(\mathcal{A}(P_v)) &=\inv_w\Big((a,g(t_v))+(a,\sigma(g)(t_v))\Big) \\
    &=\inv_w\Big(a,g(t_v)\sigma(g)(t_v)\Big) \\
    &=\inv_v((a,c)).
    \end{align*}
    Now, picking $(P_w)\in X(\A_{F})$ as above, we have $$\sum_{w\in \Omega_{F}} \inv_w\Big(\mathcal{A}(P_w)\Big)=\sum_{\substack{w\in \Omega_{F} \\ [F_w\colon k_v]=2}}0+\sum_{\substack{w\in \Omega_{F} \\ [F_w\colon k_v]=1}} \inv_w\Big(\mathcal{A}(P_w)\Big)=\sum_{\substack{v\in \Omega_{k} \\ v \; \text{splits in} \; F}} \inv_v(a,c)$$ Now, if 
    $$\sum_{\substack{v\in \Omega_{k} \\ v \; \text{splits in} \; F}} \inv_v(a,c)=0\in \Q/\Z$$
    then $X(\A_F)^{\Br (X_F)}\neq \emptyset$ and it follows that $X(F)\neq \emptyset$. Since $F/k$ is the only even degree extension over which the set $S$ decomposes as two points of degree $2$, it follows from \Cref{cor: not surjective case} that if $L/k$ is any even degree extension, then $X(L)\neq \emptyset \Leftrightarrow X(\A_L)\neq \emptyset$.
\end{proof}

\section*{Declarations}
    \subsection*{Conflict of Interest Statement} The authors are not aware of any financial or non-financial interests related to this publication.
    \subsection*{Data Availability Statement} No data was generated or analyzed for this paper, so data sharing is not applicable.

\end{document}